\documentclass[preprint,12pt]{elsarticle}

\usepackage{hyperref,amsmath,amsfonts,mathrsfs}
\usepackage{graphicx}
\usepackage{amsfonts}
\usepackage{amsbsy}
\usepackage{amssymb}
\usepackage{amsmath,tikz,pgfplots}
\usepackage{tikz}
\usetikzlibrary{arrows.meta}
\usepackage[margin=1in]{geometry}
\usetikzlibrary{decorations.pathmorphing}
\usepackage{tikz}
\numberwithin{equation}{section}
\usetikzlibrary{arrows.meta, positioning}
\usepackage{graphics}
\usepackage{graphicx}
\usepackage{amsfonts}
\usepackage{amsbsy}
\usepackage{mathabx}
\usepackage{amssymb}
\usepackage{amsmath,tikz,pgfplots}
\usetikzlibrary{decorations.pathmorphing}
\usepackage{graphics}
\def\bpsp{\begin{pspicture}}
\def\epsp{\end{pspicture}}

\newtheorem{theorem}{Theorem}[section]
\newtheorem{remark}[theorem]{Remark}
\newtheorem{example}[theorem]{Example}
\newtheorem{lemma}[theorem]{Lemma}
\newtheorem{corollary}[theorem]{Corollary}
\newtheorem{definition}[theorem]{Definition}
\newtheorem{proposition}[theorem]{Proposition}
\newtheorem{note}{Note}
\newtheorem{case}{Case}
\newtheorem{conjecture}{Conjecture}
\newtheorem{question}{Question}

\newcommand{\bea}{\begin{eqnarray}}
\newcommand{\eea}{\end{eqnarray}}
\newcommand{\beq}{\begin{eqnarray*}}
\newcommand{\eeq}{\end{eqnarray*}}

\def\m4{\mbox{\rm ~(mod $4$)}}

\def \bd{\begin{definition}}
\def \ed{\end{definition}}
\def \bqu{\begin{question}}
\def \equ{\end{question}}
\def \bcc{\begin{conjecture}}
\def \ecc{\end{conjecture}}
\def \bt{\begin{theorem}}
\def \et{\end{theorem}}
\def \bl{\begin{lemma}}
\def \el{\end{lemma}}
\def \bc{\begin{corollary}}
\def \ec{\end{corollary}}
\def \be{\begin{equation}}
\def \ee{\end{equation}}
\def \ben{\begin{enumerate}}
\def \een{\end{enumerate}}
\def \ba{\begin{array}}
\def \ea{\end{array}}
\def \bp{\begin{proposition}}
\def \ep{\end{proposition}}
\def \bx{\begin{example}}
\def \ex{\end{example}}
\def \br{\begin{remark}}
\def \er{\end{remark}}
\def \bdsc{\begin{description}}
\def \edsc{\end{description}}

\def \bn{\begin{case}}
\def \en{\end{case}}
\def \bnt{\begin{note}}
\def \ent{\end{note}}
\def\1{1\!\!1}

\def\mm2{\mbox{\rm ~(mod $2$)}}
\def\m4{\mbox{\rm ~(mod $4$)}}

\def\qed{\nolinebreak\hfill\rule{.2cm}{.2cm}\par\addvspace{.5cm}}

\def\m{\mu}

\def\1{\textbf{1}}
\def\0{\textbf{0}}

\journal{ABC}
\usetikzlibrary{arrows.meta, arrows}
\begin{document}

\begin{frontmatter}



\title{On singular values and trace norm of signed digraphs.}
\author{Mushtaq A. Bhat$^{a}$}
\author{Peer Abdul Manan$^{b}$}

\address{Department of  Mathematics,  National Institute of Technology Srinagar, Jammu and Kashmir India-190006}
\address {$^{a}$mushtaqab@nitsri.ac.in;~~~$^{b}$\text{mananab214@gmail.com}}

\begin{abstract}
  Let $S = (D, \sigma)$ be a signed digraph, where $D=(\mathcal{V},\mathcal{A})$ is the underlying digraph of $S$ and $\sigma:\mathcal{A}\rightarrow\{-1,+1\}$ is a sign function. In this paper, we study the adjacency singular values of signed digraphs. We provide a necessary and sufficient condition for the singular values of a digraph to remain invariant under signing. We study rank of signed digraphs and determine signed digraphs with rank one. As an application of this result, we obtain a lower bound for the trace norm of signed digraphs and characterize the extremal signed digraphs.

\end{abstract}

\vskip 0.2 true cm

\begin{keyword} Rank, Signed digraph, Singular value, Switching, Trace norm.
\vskip 0.2 true cm


$MSC$: 05C20, 05C50, 15A60 %% MSC codes here, in the form: \MSC code \sep code
%% or \MSC[2008] code \sep code (2020 is the default)

\end{keyword}

\end{frontmatter}

\section{\bf Introduction}

A signed digraph is a pair 
$S = (D, \sigma)$, where $D = (\mathcal{V}, \mathcal{A})$ is the underlying digraph and 
$\sigma : \mathcal{A} \to \{-1, 1\}$ is a sign function. All our signed digraphs are simple i.e., free from loops and parallel signed arcs. An arc from a vertex $u$ to a vertex $v$ is denoted by $uv$. 
An arc $uv$ is said to be positive or negative according as $\sigma(uv) = +1$ or $\sigma(uv) = -1$. The sign of a signed digraph $S$ is defined as the product of the signs 
of all its arcs. The sign of a signed subdigraph  $H$ of signed digraph $S$ is the restriction of the sign function to $H$. In a digraph $D$, a directed path of order $k$ is denoted by $\overrightarrow{P}_k$, consists of $k$ vertices say $v_1, v_2, \dots, v_k$ and arcs $v_1v_2, v_2v_3,\dots, v_{k-1}v_k$. In a digraph $D$, a directed cycle of order $k$ is denoted by $\overrightarrow{C}_k$, consists of $k$ vertices say $v_1, v_2, \dots, v_k$ and arcs $v_1v_2, v_2v_3,\dots, v_{k-1}v_k, v_kv_1$. The underlying graph $G$ of a digraph $D$ is obtained by removing direction of each arc of $D$ and replacing each directed cycle of length $2$ by an edge in $G$. Moreover, if $P_k$ is path in underlying graph $G$ of digraph $D$ but is not the directed path $\overrightarrow{P}_k$ in $D$, we call it a semi-directed path. Also, if $C_k$ is cycle in underlying graph $G$ of digraph $D$ but is not the directed cycle $\overrightarrow{C}_k$ in $D$, we call it a semi directed cycle. A directed cycle (semi directed cycle) in a signed digraph is said to be positive or negative according as its sign is positive or negative respectively. A signed digraph is said to be balanced if each of its directed cycle and semi-directed cycle is positive. A positive directed cycle of order $n$ is denoted by $\overrightarrow{C}_n^+$ and a negative 
directed cycle of order $n$ is denoted by $\overrightarrow {C}_n^{-}$.\\ Let $S$ be a signed digraph of order $n$ with vertices $v_1, v_2, \dots, v_n$. Then the adjacency matrix of $S$, denoted by 
$A(S) = [a_{ij}]_{n \times n}$, is the square matrix of order $n$ with $a_{ij}=\sigma(v_i, v_j)$ if there is an arc from $v_i$ to $v_j$ and zero, otherwise. The characteristic polynomial of $A(S)$, $\phi_{A(S)}(\lambda) = \det(A(S) - \lambda I_n)$ is known as the characteristic polynomial of $S$ and the eigenvalues of $A(S)$ are known as the eigenvalues of $S$. A signed graph $\Sigma$ is a pair $\Sigma=(G(\mathcal{V},\mathcal{E}),\sigma)$, where $\sigma:\mathcal{E}\rightarrow\{-1,+1\}$ is a sign function and $G=G(\mathcal{V},\mathcal{E})$ is the underlying graph of $\Sigma$. 
A signed digraph $S$ is said to be symmetric if, for every arc $uv \in \mathcal{A}(S)$, 
the arc $vu \in \mathcal{A(S)}$ also exists with the same sign, where $u, v \in \mathcal{V}(S)$. 
There exists a one-to-one correspondence between signed graphs and symmetric 
signed digraphs given by $\Sigma \longleftrightarrow \overleftrightarrow{\Sigma}$, where $\overleftrightarrow{\Sigma}$ denotes the signed digraph obtained from signed graph $\Sigma$ with the same vertex set as that of the signed graph $\Sigma$ and each signed edge in $\Sigma$ is replaced by a pair of 
symmetric arcs $uv$ and $vu$, both carrying the same sign as that of the edge $uv$. 
Under this correspondence, a signed graph can be identified with a symmetric signed digraph. For a vertex $u$ in a digraph $D$, the outdegree $d^+(u)$ is the number 
of arcs leaving $u$, and the indegree $d^-(u)$ is the number of 
arcs entering $u$. The indegree and outdegree in signed digraph are those of its underlying digraph.\\ 
Let $S=(D(\mathcal{V},\mathcal{A}),\sigma)$ be a signed digraph and $Y \subset \mathcal{V}$. Switching $S$ by $Y$ is the signed digraph obtained from $S$ by changing signs of all arcs between $Y$ and $\mathcal{V}\setminus Y$. For switching in signed graphs and undefined terminology see \cite{bsp,tz}. The switched signed digraph thus obtained is denoted by $S^Y$. Note that switching is an equivalence relation on the signings of digraph and its equivalence classes are known as switching classes. An alternative way to define switching is by means of a function $\chi:\mathcal{V}\rightarrow \{-1,+1\}$ such that $\sigma^{\chi}(uv)=\chi(u)\sigma(uv)\chi(v)$. In terms of matrices if $A$ is the adjacency matrix of a signed digraph $S$ and $A^{\chi}$ be the adjacency matrix of switched signed digraph $S^{\chi}$, then $A^{\chi}=M^{-1}AM$, where $M$ is the signature matrix (i.e., the diagonal matrix with diagonal entries from the set $\{-1,+1\}$) with $M_{ii}=\chi(v_i)$. Note that $M=M^T=M^{-1}$. Two signed digraphs are said to be switching isomorphic if one is isomorphic to switching of other. For example in undirected case see \cite{tz}.\\
The singular values of a signed digraph $S$ are singular values of its adjacency matrix $A$ i.e., the positive square roots of eigenvalues of $AA^T$. If a signed digraph $S$ of order $n$ has $k\le n$ distinct singular values $\sigma_1, \sigma_2, \dots, \sigma_k$ with respective multiplicities $m_1, m_2, \dots, m_k$, then we write the set of singular values as $\{\sigma_1^{(m_1)}, \sigma_2^{(m_2)}, \dots, \sigma_k^{(m_k)}\}$ and call it singular spectrum of $S$. We denote the singular spectrum of a signed digraph $S$ by $\mathcal{S}(S)$. Throughout, we assume the order of singular values as $\sigma_1(S)\ge\sigma_2(S)\ge \dots \ge\sigma_n(S)\ge0$. The singular value $\sigma_1(S)$ is called the spectral norm of $S$. The trace norm of a signed digraph $S$ with singular values $\sigma_1(S), \sigma_2(S), \dots, \sigma_n(S)$ is denoted by $\|A\|_*$ and is defined as $$\|A\|_*=\sum_{i=1}^{n}\sigma_i(S).$$
The notion of energy in signed graphs was given by Germina et al.~\cite{ghz} and they defined the energy of a signed graph as the sum of absolute values of signed graph eigenvalues. The concept of energy was extended to signed digraphs by Pirzada and Bhat ~\cite{pb} and they defined the energy of a signed digraph as the sum of absolute values of real parts of signed digraph eigenvalues. It is clear that for signed graphs trace norm becomes the energy. Agudelo, Rada \cite{ar} obtained lower bounds for the trace norm of digraphs. Agudelo, Pe{\~n}a and Rada obtained the extremal values of trace norm over oriented trees. For more on trace norm and energy see \cite {amr,bp,bsp,gc,mr,psg,rgw}.  In this paper, we study the energy of signed digraphs by using trace norm definition.\\
By $\mathcal{T}(n)$, we denote the set of directed trees of order $n$.  Let $\overrightarrow{K}_{p,q}$ denote the complete bipartite digraph with partite sets say $P=\{u_1,u_2,\dots,u_p\}$ and $Q=v_1,v_2,\dots, v_q$ and  $\mathcal{A}(\overrightarrow{K}_{p,q})=\{u_iv_j:i=1,2,\dots, p ~~\text{and}~~ j=1,2,\dots,q\}$. The rest of the paper is oragnized as follows. Section $1$ gives introduction of the problem and some definitions used in the sequel. In section $2$, we state the preliminary results used throughout the paper. In section $3$, we compute the singular values of some basic signed digraphs. We provide the definition of bipartite signed graph associated to a signed digraph and using this we provide a necessary and sufficient condition for two signatures $\tau$ and $\sigma$ to be equivalent by means of signature matrices. We completely characterize signed digraphs $S=(D,\sigma)$ which share same singular spectrum with their respective underlying digraphs. This result gives families of switching non-isomorphic signed digraphs with same singular spectrum. In section $4$, we characterize signed digraphs with rank one and as an application of this result, we obtain a lower bound for the trace norm of signed digraphs and characterize the extremal signed digraphs.

\section{\bf Preliminaries}
In this section we recall some known results that will be used in sequel.\\
The following result is known as the interlacing property of singular values.
\begin{lemma} \cite{hjtop} Let $A\in M_n(\mathbb{C})$ and let $A_r$ denote a submatrix of $A$ obtained by deleting $r$ rows and /or columns from $A$. Then $$\sigma_k(A)\ge \sigma_k(A_r)\ge \sigma_{k+r}(A), ~~~k=1,2,\dots,n,$$
where for $X\in M_n(\mathbb{C})$, we set $\sigma_j(X)=0$ if $j>n$. In particular, if $A\in M_n(\mathbb{C})$ and $B$ is a principal submatrix [or any submatrix] of $A$ of order $r$. Then $$\sigma_k(A)\ge \sigma_k(B)\ge \sigma_{k+2(n-r)}(A),$$
	where $k=1,2,\dots,n.$ 
\end{lemma}
A vertex $v$ in a digraph is said to be a sink or source vertex if $d^+(v)=0$ or $d^-(v)=0$ respectively. A digraph is said to be sink-source if each of its vertex is either a sink or a source vertex. The following result holds in this direction.

\begin{lemma} \label{1.2}\cite{mr} Let $G$ be a graph. Then $G$ has a sink-source orientation if and only if $G$ is bipartite.\end{lemma}
\begin{lemma}\cite{hjtop} Let $C$ be a circulant matrix of order $n$ with first row $[\alpha_0,\alpha_1,\dots,\alpha_{n-1}]$. Then the eigenvalues $\lambda_k$ of $C$ are $$\lambda_k=\sum\limits_{j=0}^{n-1}\alpha_j\omega^{jk},~~\text{where}~~ k=0,1,2,\dots,n-1~~ \text{and}~~ \omega=\exp\left(\frac{2\pi\iota}{n}\right) ~is~ ~n^{\text{th}}~ root~ of~ unity.$$
\end{lemma}
\begin{lemma} \cite{bp} Let $B$ a bipartite signed graph. Then if $\lambda$ is an eigenvalue of $B$, then $-\lambda$ is also an eigenvalue of $B$ with same multiplicity.
\end{lemma}
\begin{proposition} In a signed digraph, the switching preserves the sign of a directed cycle or semi directed cycle.
\end{proposition}
{\bf Proof}. Let $S$ be sidigraph of order $n\ge 2$ and let it has a cycle $\overrightarrow{C}_k$ of order $k$ with cyclic vertices being $v_1,v_2,\dots,v_k$, where $2\le k\le n$. Then for any switching function $\chi$ on $S$,  $\text{sgn}_{S^{\chi}}(\overrightarrow{C}_k)=\text{sgn}_{S}(\overrightarrow{C}_k)(\chi(v_1))^2(\chi(v_2))^2\dots(\chi(v_k))^2=\text{sgn}_{S}(\overrightarrow{C}_k)$, where $\text{sgn}_{S}(\overrightarrow{C}_k)$ stands for sign of $\overrightarrow{C}_k$ in $S$. A similar proof holds for semi directed cycle.\qed

Recall that a signed digraph is said to be empty or discrete if its has no arcs. The following result follows from singular value decomposition of a matrix.
\begin{lemma} \cite{bm} A signed digraph $S$ of order $n$ is discrete if and only if $\mathcal{S}(S)=\{0^{(n)}\}$.
\end{lemma}
\begin{lemma}\cite{htw}Two signed graphs $\Sigma_1$ and $\Sigma_2$ with the same underlying graph are switching equivalent if and only if there exist signature matrix $M$ such that $M^{-1}A(\Sigma_1)M=A(\Sigma_2)$, where $A(\Sigma_i)$, $i=1,2$ is the adjacency matrix of $\Sigma_i$.
\end{lemma}
\begin{lemma}\cite{a} A signed graph $\Sigma=(G,\sigma)$ is balanced if and only if $\Sigma$ and $G$ have same adjacency eigenvalues including multiplicities.
\end{lemma}
\section{\bf Invariance of singular values of a digraph under signing}\label{sec2}
We begin this section with examples which also provide motivation for our results.
\begin{example}
\upshape Let $A$ denote the adjacency matrix of the signed directed path $(\overrightarrow{P}_n,\sigma)$. Then
\[
AA^T = 
\begin{bmatrix}
	I_{n-1} & {\bf 0}_{(n-1)\times 1} \\
	{\bf 0}_{1 \times (n-1)} & 0_{1 \times 1}
\end{bmatrix}.
\]
Consequently, $\mathcal{S}(\overrightarrow{P}_n,\sigma)=\{1^{(n-1)},0\}$.
\end{example}
\begin{example}
\upshape Let $A$ be the adjacency matrix of a signed directed cycle $(\overrightarrow{C}_n, \sigma)$. Then $AA^T=I_n$. Consequently, $\mathcal{S}(\overrightarrow{C}_n, \sigma) =\{1^{(n)}\}$, irrespective of sign of $\overrightarrow{C}_n$.
\end{example}
By $\widehat{C}_n$ we denote the sink source semidirecetd cycle of even order $n$. For example $\widehat{C}_{10}$ is shown in Figure 1. 
\begin{example}
\upshape Let $\widehat{C}_n$ be the sink source cycle of order $n$. Then by Lemma $2.2$, $n$ is even say $n=2k$, where $k\ge 2$. We label the source vertices by {$1,2,\dots,k$} and sink vertices by {$k+1,k+2,\dots, 2k$}. Let $A$ be the adjacency matrix of $\widehat{C}_n$. Then for $k=2$
\[
	AA^T = 
	\begin{bmatrix}
		{B_2} & {\bf 0}_{2\times 2} \\
		{\bf 0}_{2\times 2} & {\bf 0}_{2\times 2}
	\end{bmatrix}, 
		\]
where $B_2=\text{Circ}~[2~~ 2]$. Clearly, $\mathcal{S}(\widehat{C}_4)=\{0^{(3)}, 2\}$.\\
For $k=3$, we see that 
\[
AA^T = 
\begin{bmatrix}
	{B_3} & {\bf 0}_{3\times 3} \\
	{\bf 0}_{3\times 3} & {\bf 0}_{3\times 3}
\end{bmatrix},
\]
where $B_3=\text{Circ}~[2~~ 1~~ 1]$. Using Lemma $2.3$, it is easy to see that $\mathcal{S}(\widehat{C}_6)=\{0^{(3)}, 1^{(2)}, 2\}$. \\ 
For $k\ge 4$, we see
\[
AA^T = 
\begin{bmatrix}
	{B_k} & {\bf 0}_{k\times k} \\
	{\bf 0}_{k\times k} & {\bf 0}_{k\times k}
\end{bmatrix},
\]
where $B_k=\text{Circ}~[2~~ 1~~0,~~\dots,~~0,~~1]$ is circulant matrix of order $k$. \\
By Lemma $2.3$, the eigenvalues of $B_k$ are $$\lambda_{j}(B_k)=2+\omega^j+\omega^{-j}=2+2\cos\left(\frac{2\pi j}{k}\right)=4\cos^2\left(\frac{\pi j}{k}\right),$$
where $\omega=\exp(\frac{2\pi i}{k})$, $i=\sqrt{-1}$ and $j=0,1,2,\dots,k-1$.\\
Consequently, $\mathcal{S}(\widehat{C}_{2k})=\{0^{(k)}, 2|\cos(\frac{\pi j}{k})|$, $j=0,1,2,\dots,k-1\}$.
\end{example}
\begin{example}
\upshape Let $\widehat{C}_n^{-}$ denote the sink source signed cycle obtained from $\widehat{C}_n$ by making the arc $(1,k+1)$ negative. Let $A$ be the adjacency matrix of $\widehat{C}_n^{-}$. Then 
For $k=2$, we have 
\[
AA^T = 
\begin{bmatrix}
	{2I_2} & {\bf 0}_{2\times 2} \\
	{\bf 0}_{2\times 2} & {\bf 0}_{2\times 2}
\end{bmatrix}. 
\]
Therefore, $\mathcal{S}(\widehat{C}_4^{-})=\{0^{(2)}, \sqrt{2}^{(2)}\}$.\\
For $k\ge 3$, it is easy to verify that 
\[
AA^T = 
\begin{bmatrix}
	{B_k^-} & {\bf 0}_{k\times k} \\
	{\bf 0}_{k\times k} & {\bf 0}_{k\times k}
\end{bmatrix}, 
\]
where $B_k^-=2I_k+A(C_k^-)$, with $A(C_k^-)$ being the adjacency matrix of negative undirected signed cycle of order $k$ with exactly one negative edge. The eigenvalues \cite[Lemma 2.6]{bpeq} of $A(C_k^-)$ are \\$2\cos\left(\frac{(2j+1)\pi}{k}\right)$, where $j=0,1,2,\dots, k-1$. The eigenvalues of $B_k^-$ are $2+2\cos\left(\frac{(2j+1)\pi}{k}\right)=4\cos^2\left(\frac{(2j+1)\pi}{2k}\right)$, $j=0,1,2,\dots,k-1$.\\
Consequently, $\mathcal{S}(\widehat{C}_{2k}^{-})=\{0^{(k)}, 2\big|\cos\left(\frac{(2j+1)\pi}{2k}\right)\big|$, $j=0,1,2,\dots,k-1\}$.
\end{example}
From examples $3.3$ and $3.4$, we see that $\widehat{C}_n$ and $\widehat{C}_n^{-}$ do not have the same singular spectrum, because $2$ is a singular value of $\widehat{C}_{2k}^{+}$ but $2$ is not a singular value of $\widehat{C}_{2k}^{-}$ as $0<\frac{(2j+1)\pi}{2k}<\pi$ for $j=0,1,2,\dots,k-1$ but $\cos{\frac{(2j+2)\pi}{2k}}\in(-1,1)$.\\
Next we recall from \cite{bhm} the definition of the bipartite support graph $H(D)$ associated to digraph $D$.

\begin{definition}
Let $D$ be a digraph of order $n$ and let $\mathcal{V}=\{v_1, v_2, \dots, v_n\}$ be the vertex set of $D$. The bipartite support graph of $D$ denoted by $H(D)$ is the graph with vertex set $\mathcal{V}_H=\mathcal{V}\cup\mathcal{V}^{\prime}$, where $\mathcal{V}^{\prime}=\{v_1^{\prime}, v_2^{\prime}, \dots, v_n^{\prime}\}$. The adjacency in $H(D)$ is defined as $v_iv_j^{\prime}$ is an arc in $H(D)$ if and only if $v_iv_j$ is an arc in $D$.
\end{definition}
It is clear from the definition that if $D$ has $n$ vertices and $m$ arcs, the $H(D)$ has $2n$ vertices and $m$ edges.\\ We associate a bipartite signed graph to a signed digraph as follows.
\begin{definition}
Let $S=(D,\sigma)$ be a signed digraph with $n$ vertices. Then the bipartite support signed graph of $S$, which we denote by $H(S)$ is the bipartite signed graph whose underlying bipartite graph is $H(D)$ and sign of an edge $\sigma(v_iv_j^{\prime})=\sigma(v_iv_j)$. We write this by $H(S)=(H(D),\sigma)$.
\end{definition}
\begin{example}
Consider the signed digraphs $S_1$ and $S_2$ shown in Figure 2. Their bipartite support signed graphs are $H(S_1)$ and $H(S_2)$. The solid arrows/edges represent positive arcs/edges and the dotted arrows/edges denote negative arcs/edges. 
\end{example}
\begin{figure}
	\centering
	\begin{tikzpicture}[ >=stealth, ->, every node/.style={circle,draw,inner sep=1.5pt} ] \node (1) at (0,1) {1}; \node (6) at (1.2,1) {6}; \node (2) at (2.4,1) {2}; \node (7) at (3.6,1) {7}; \node (3) at (4.8,1) {3}; \node (8) at (4.8,0) {8}; \node (4) at (3.6,-1){4}; \node (9) at (2.4,-1){9}; \node (5) at (1.2,-1){5}; \node (10) at (0,-1) {10}; \draw (1)--(6); \draw (2)--(6); \draw (2)--(7); \draw (3)--(7); \draw (3)--(8); \draw (4)--(8); \draw (4)--(9); \draw (5)--(9); \draw (5)--(10); \draw (1)--(10); \end{tikzpicture}
	\caption{Semidirected cycle $\widehat{C_{10}}$}
\end{figure}
\begin{figure}[htbp]
	\centering
	\begin{tikzpicture}[
		vertex/.style={circle, draw,  inner sep=0pt, minimum size=4mm, font=\small},
		edge/.style={-{Stealth[scale=1.2]}, thick, shorten >=1pt},
		uedge/.style={thick} % Undirected edge style
		]
		
		% ==========================================
		% GRAPH S_1 (Top-Left)
		% ==========================================
		\begin{scope}[shift={(0,0)}]
			% Triangle vertices (v2 at the top)
			\node[vertex] (v2) at (0, 2)     {$v_2$};
			\node[vertex] (v3) at (1.73, 0)  {$v_3$};
			\node[vertex] (v1) at (-1.73, 0) {$v_1$};
			
			% Incoming vertices to v2
			\node[vertex] (v4) at (-1.5, 3.2) {$v_4$};
			\node[vertex] (v5) at (1.5, 3.2)  {$v_5$};
			
			% Directed edges (v2 -> v3 is dotted)
			\draw[edge] (v1) -- (v2);
			\draw[edge, dotted] (v2) -- (v3);
			\draw[edge] (v3) -- (v1);
			\draw[edge] (v4) -- (v2);
			\draw[edge] (v5) -- (v2);
			
			% Label for S_1
			\node at (0, -0.9) {$S_1$};
		\end{scope}
		
		% ==========================================
		% GRAPH S_2 (Top-Right: 4-cycle / square)
		% ==========================================
		\begin{scope}[shift={(7.5, 0.5)}]
			% Vertices arranged cyclically in a square
			\node[vertex] (u1) at (0, 2) {$v_1$};
			\node[vertex] (u2) at (2, 2) {$v_2$};
			\node[vertex] (u3) at (2, 0) {$v_3$};
			\node[vertex] (u4) at (0, 0) {$v_4$};
			
			% Directed perimeter edges
			\draw[edge] (u1) -- (u2);
			\draw[edge, dotted] (u1) -- (u4);
			\draw[edge, dotted] (u3) -- (u2);
			\draw[edge] (u3) -- (u4);
			
			% Label for S_2
			\node at (1, -1.4) {$S_2$};
		\end{scope}
		
		% ==========================================
		% GRAPH H(S_1) (Bottom-Left: Undirected)
		% ==========================================
		\begin{scope}[shift={(-0.5, -6.5)}]
			% Left column: unprimed vertices v_1 to v_5
			\node[vertex] (h1_v1) at (0, 4) {$v_1$};
			\node[vertex] (h1_v2) at (0, 3) {$v_2$};
			\node[vertex] (h1_v3) at (0, 2) {$v_3$};
			\node[vertex] (h1_v4) at (0, 1) {$v_4$};
			\node[vertex] (h1_v5) at (0, 0) {$v_5$};
			
			% Right column: primed vertices v_1' to v_5'
			\node[vertex] (h1_pv1) at (3, 4) {$v_1'$};
			\node[vertex] (h1_pv2) at (3, 3) {$v_2'$};
			\node[vertex] (h1_pv3) at (3, 2) {$v_3'$};
			\node[vertex] (h1_pv4) at (3, 1) {$v_4'$};
			\node[vertex] (h1_pv5) at (3, 0) {$v_5'$};
			
			% Undirected edges between unprimed and primed vertices
			\draw[uedge] (h1_v1) -- (h1_pv2);
			\draw[uedge, dotted] (h1_v2) -- (h1_pv3);
			\draw[uedge] (h1_v3) -- (h1_pv1);
			\draw[uedge] (h1_v4) -- (h1_pv2);
			\draw[uedge] (h1_v5) -- (h1_pv2);
			
			% Label for H(S_1)
			\node at (1.5, -1) {$H(S_1)$};
		\end{scope}
		
		% ==========================================
		% GRAPH H(S_2) (Bottom-Right: Undirected)
		% ==========================================
		\begin{scope}[shift={(7, -6)}]
			% Left column: unprimed vertices v_1 to v_4
			\node[vertex] (h2_v1) at (0, 3) {$v_1$};
			\node[vertex] (h2_v2) at (0, 2) {$v_2$};
			\node[vertex] (h2_v3) at (0, 1) {$v_3$};
			\node[vertex] (h2_v4) at (0, 0) {$v_4$};
			
			% Right column: primed vertices v_1' to v_4'
			\node[vertex] (h2_pv1) at (3, 3) {$v_1'$};
			\node[vertex] (h2_pv2) at (3, 2) {$v_2'$};
			\node[vertex] (h2_pv3) at (3, 1) {$v_3'$};
			\node[vertex] (h2_pv4) at (3, 0) {$v_4'$};
			
			% Undirected edges in H(S_2)
			\draw[uedge] (h2_v1) -- (h2_pv2);
			\draw[uedge, dotted] (h2_v1) -- (h2_pv4);
			\draw[uedge, dotted] (h2_v3) -- (h2_pv2);
			\draw[uedge] (h2_v3) -- (h2_pv4);
			
			% Isolated vertices: v_2, v_4, v_1', and v_3'
			
			% Label for H(S_2)
			\node at (1.5, -1.5) {$H(S_2)$};
		\end{scope}
		
	\end{tikzpicture}
	\caption{The signed digraphs $S_1, S_2$ and their associated bipartite signed graphs $H(S_1), H(S_2).$}
	\label{fig:digraphs_and_bipartite_graphs}
\end{figure}

\begin{lemma}
Let $S=(D,\sigma)$ be a signed digraph of order $n$. Then $\sigma_1, \sigma_2, \sigma_3, \dots, \sigma_n$ are the singular values of $S$ if and only if  $\{\pm\sigma_1,\pm\sigma_2,\pm\sigma_3,\dots, \pm\sigma_n\}$ are the eigenvalues of $H(S)=(H(D),\sigma)$.
\end{lemma}
{\bf Proof.} Let $v_1,v_2,\dots,v_n$ be the $n$ vertices of $S$. Then by definition $H(S)$ has $2n$ vertices say $v_1,v_2,\dots,v_n, v_1^{\prime},v_2^{\prime},\dots, v_n^{\prime}$ and $v_iv_j^{\prime}$ is an edge in $H(S)$ if and only if $v_iv_j$ is an arc in $S$ and both share the same sign. With suitable labelling, the adjacency matrix of $H(S)$ can be written as
\[
A_H = A(H(S))= 
\begin{bmatrix}
	{\bf 0}_{n\times n} & A(S) \\
	{A^T(S)} & {\bf 0}_{n\times n}
\end{bmatrix}, 
\]
so that 

\begin{equation}
A_H{A^T_H} = A^2_H= 
\begin{bmatrix}
	A(S){A^T(S)} & {\bf 0}_{n\times n} \\
	{\bf 0}_{n\times n} & {A^T(S)}A(S)
\end{bmatrix},
\end{equation} 

Let $\sigma_1, \sigma_2, \sigma_3, \dots, \sigma_n$ are the singular values of $S$, then the eigenvalues of $A_H{A^T_H}={A^2_H}$ are  $\{\sigma_1^2,\sigma_1^2,\sigma_2^2,\sigma_2^2,\dots,\sigma_n^2,\sigma_n^2\}$ and since $H(S)$ is bipartite by Lemma $2.4$, the eigenvalues of $H(S)$ are $\{\pm\sigma_1,\pm\sigma_2,\dots,\pm\sigma_n\}$.\\
Conversely, suppose that the eigenvalues of $H(S)=(H(D),\sigma)$ are $\{\pm\sigma_1,\pm\sigma_2,\pm\sigma_3,\dots, \pm\sigma_n\}$. By $(3.1)$, we see that singular values of $S$ must be $\sigma_1, \sigma_2, \sigma_3, \dots, \sigma_n$.\qed

Let $\tau$ and $\sigma$ be two signatures on a digraph $D$ of order $n$  let $A_{\tau}$ and $A_{\sigma}$ be the adjacency matrices of associated signed digraphs respectively. The next result relates the $A_{\tau}$ and $A_{\sigma}$ by means of signature matrices.
\begin{theorem}
Let $D$ be a digraph of order $n$. Then for signatures $\tau$ and $\sigma$ on $D$ with respective adjacency matrices $A_{\tau}$ and $A_{\sigma}$ there exists signature matrices $R$ and $C$ of order $n$ such that $RA_{\tau}C=A_{\sigma}$ if and only if the associated bipartite support signed graphs $(H(D),\tau)$ and $(H(D),\sigma)$ are switching equivalent. 
\end{theorem}

{\bf Proof.} Assume first that there exists signature matrices $R$ and $C$ such that $RA_{\tau}C=A_{\sigma}$. We need to show $(H(D),\tau)$ and $(H(D),\sigma)$ are switching equivalent.\\
Note that 

\[
A(H(D),\tau)=
\begin{bmatrix}
	{\bf 0}_{n\times n} &A_{\tau}\\
	{A^T_{\tau}} & {\bf 0}_{n\times n}
\end{bmatrix}
\]
and 
\[
A(H(D),\sigma)=
\begin{bmatrix}
	{\bf 0}_{n\times n} &A_{\sigma}\\
	{A^T_{\sigma}} & {\bf 0}_{n\times n}
\end{bmatrix}.
\]
As $R$ and $C$ are signature matrices so is the matrix $M$ defined as 
\[
M=\begin{bmatrix}
	R & {\bf 0}_{n\times n}\\
	{\bf 0}_{n\times n} & C
\end{bmatrix}=M^{-1}
\]
We have 
\[
\begin{aligned}
	&\begin{bmatrix}
		R & {\bf 0}_{n\times n}\\
		{\bf 0}_{n\times n} & C
	\end{bmatrix}
	\begin{bmatrix}
		{\bf 0}_{n\times n} & A_{\tau}\\
		{A^T_{\tau}} & {\bf 0}_{n\times n}
	\end{bmatrix}
	\begin{bmatrix}
		R & {\bf 0}_{n\times n}\\
		{\bf 0}_{n\times n} & C
	\end{bmatrix}
	\\
	&\quad =
	\begin{bmatrix}
		{\bf 0}_{n\times n} & RA_{\tau}\\
		C{A^T_{\tau}} & {\bf 0}_{n\times n}
	\end{bmatrix}
	\begin{bmatrix}
		R & {\bf 0}_{n\times n}\\
		{\bf 0}_{n\times n} & C
	\end{bmatrix}
	\\
	&\quad =
	\begin{bmatrix}
		{\bf 0}_{n\times n} & RA_{\tau}C\\
		C{A^T_{\tau}}R & {\bf 0}_{n\times n}
	\end{bmatrix}
	\\
	&\quad =
	\begin{bmatrix}
		{\bf 0}_{n\times n} & A_{\sigma}\\
		{A^T_{\sigma}} & {\bf 0}_{n\times n}
	\end{bmatrix}.
\end{aligned}
\]
Therefore the matrices \[
\begin{bmatrix}
	{\bf 0}_{n\times n} & A_{\tau}\\
	{A^T_{\tau}} & {\bf 0}_{n\times n}
\end{bmatrix}
\text{and} 
\begin{bmatrix}
	{\bf 0}_{n\times n} & A_{\sigma}\\
	{A^T_{\sigma}} & {\bf 0}_{n\times n}
\end{bmatrix}
\]
are signature similar. By Lemma $2.7$, we see $(H(D),\tau)$ and $(H(D),\sigma)$ are switching equivalent.\\
Conversely, suppose $(H(D),\tau)$ and $(H(D),\sigma)$ are switching equivalent. Then again by Lemma $2.7$, there exits a signature matrix say $M$ of order $2n$ such that $$M^{-1}A(H(D),\tau)M=A(H(D),\sigma).$$
Partition $M$ as \[
M=\begin{bmatrix}
	M_1 & {\bf 0}_{n\times n}\\
	{\bf 0}_{n\times n} & M_2
\end{bmatrix}
\]
Clearly $M_1$ and $M_2$ are signature matrices. We see that\\
\[
\begin{bmatrix}
	{M_1}^{-1} & {\bf 0}_{n\times n}\\
	{\bf 0}_{n\times n} & {M_2}^{-1}
\end{bmatrix}
\begin{bmatrix}
	{\bf 0}_{n\times n} & A_{\tau}\\
	{A_{\tau}}^{T} & {\bf 0}_{n\times n}
\end{bmatrix}
\begin{bmatrix}
	{M_1} & {\bf 0}_{n\times n}\\
	{\bf 0}_{n\times n} & {M_2}
\end{bmatrix}
=\begin{bmatrix}
	{\bf 0}_{n\times n} & A_{\sigma}\\
	{A^T_{\sigma}} & {\bf 0}_{n\times n}
\end{bmatrix},
\]
or \[
\begin{bmatrix}
	{\bf 0}_{n\times n} & M_1^{-1}A_{\tau}M_2\\
	M_2A^T_{\tau}M_1^{-1} & {\bf 0}_{n\times n}
\end{bmatrix}
=
\begin{bmatrix}
	{\bf 0}_{n\times n} & A_{\sigma}\\
	{A_{\sigma}}^{T} & {\bf 0}_{n\times n}
\end{bmatrix},
\]
which gives $${M_1}^{-1}A_{\tau}{M_2}=A_{\sigma}.$$
Put $R=M_1^{-1}$ and $C={M_2}$, the result follows.\qed

\begin{lemma}
Let $D$ be a digraph of order $n$. Let $\tau$ and $\sigma$ be two signatures on $D$ such that the corresponding adjacency matrices are related as $RA_{\tau}C=A_{\sigma}$, where $R$ and $C$ are signatue matrices. Then $A_{\tau}$ and $A_{\sigma}$ have same singular spectrum.
\end{lemma}
{\bf Proof.} We have 
\begin{align*}
A_{\sigma}A^T_{\sigma}&=(RA_{\tau}C)(RA_{\tau}C)^T\\
&=R^{-1}A_{\tau}A^T_{\tau}R
\end{align*}
by using $R=R^{-1}$ and $C=C^{-1}$.\\
Clearly, the matrices $A_{\sigma}A^T_{\sigma}$ and $A_{\tau}A^T_{\tau}$ are positive semidefinite and similar. So they share same eigenvalues including multiplicities and hence $A_{\tau}$ and $A_{\sigma}$ have same singular spectrum.\qed
In the next result, we characterize digraphs $D$ such that the $H(D)$ has a cycle.
\begin{theorem}
Let $D$ be a digraph of order $n$. Then for each semidirected cycle $\widehat{C}_{2k}$ for some $k\ge 2$, of $D$, there is a corresponding cycle $C_{2k}$ in associated bipartite graph $H(D)$ and conversely. Moreover, the number of semidirected cycles of the form $\widehat{C}_{2k}$ in $D$ is equal to the number of cycles in the bipartite graph $H(D)$.
\end{theorem}
{\bf Proof.} We first assume that $D$ has a semidirected cycle $\widehat{C}_{2k}$ for some $k\ge 2$. Let $\mathcal{V}=\{v_1,v_2,v_3,\dots v_{n-1},v_n\}$ be the vertex set of $D$. We assume $\mathcal{V}(\widehat{C}_{2k})=\{v_1,v_2,v_3,\dots,v_{k-1},v_k,v_{k+1},\\v_{k+2},\dots, v_{2k-1},v_{2k}\}$ and assume labeling of vertices of $\widehat{C}_{2k}$ as mentioned in Example $3.3$. By definition, in $H(D)$ $v_1$ is adjacent to ${v^{\prime}_{k+1}}$ and $v^{\prime}_{2k}$, $v_2$ is adjacent to ${v^{\prime}_{k+1}}$ and ${v^{\prime}_{k+2}}$, $v_3$ is adjacent to ${v^{\prime}_{k+2}}$ and ${v^{\prime}_{k+3}}$, and so on $v_k$ is adjacent to ${v^{\prime}_{2k-1}}$ and ${v^{\prime}_{2k}}$. Thus producing a cycle $C_{2k}:v_1{v^{\prime}_{k+1}}v_2{v^{\prime}_{k+2}}\dots v_k{v^{\prime}_{2k}}v_1$ in $H(D)$. It is clear that two different semidirected cycles $\widehat{C}_{2p}$ and $\widehat{C}_{2q}$ in $D$ give two different cycles $C_{2p}$ and $C_{2q}$ in $H(D)$.\\
Conversely, suppose that $H(D)$ has a cycle $C_{2k}$ of length $2k$ and let vertex set of $H(D)$ be $\mathcal{V}(H(D))=\{v_1,v_2,\dots, v_n, {v^{\prime}_{1}}, {v^{\prime}_{2}},\dots,{v^{\prime}_{n}}\}$. As $C_{2k}$ is an alternating cycle in $H(D)$, we assume that the vertex set of $C_{2k}$ be $\mathcal{V}(C_{2k})=\{v_1,{v^{\prime}_{2}}, v_3, {v^{\prime}_{4}},\dots, v_{2k-1}, {v^{\prime}_{2k}}\}$ and $C_{2k}:v_1{v^{\prime}_{2}}v_3{v^{\prime}_{4}}\dots,v_{2k-1} v^{\prime}_{2k}$ Then by definition of $H(D)$, $\widehat{C}_{2k}$ with arcs $:v_1v_2, v_1v_{2k}, v_3v_2,v_3v_4, v_5v_4,v_5v_6,\dots, v_{2k-1}v_{2k-2},v_{2k-1}v_{2k}$ is its corresponding semidirecetd cycle present in $D$. It is worth to mention that as $D$ is simple digraph, we cannot have arcs of the form $v_iv^{\prime}_i$ in $H(D)$.\\
We see that there is a bijection between the set of semidirected cycles of the form $\widehat{C}_{2k}$ in $D$ and the set of cycles of the form $C_{2k}$ in $H(D)$.\qed
We next obtain an invariance criterion for the singular values of a digraph under signing.
\begin{theorem}
Let $D$ be a digraph of order $n$. Then for any signature $\sigma$ on $D$, $D$ and the signed digraph $(D,\sigma)$ have same singular spectrum if and only if either $(D,\sigma)$ does not contain $\widehat{C}_{2k}$ type signed subdigraph or each signed subdigraph of $(D,\sigma)$ of the type $\widehat{C}_{2k}$ has positive sign.
\end{theorem}
{\bf Proof.} Suppose $D$ and $(D,\sigma)$ have same singular spectrum. Then by Lemma $3.8$, $H(D)$ and $(H(D),\sigma)$ must have same adjacency eigenvalues. If $(D,\sigma)$ contains a negative semidirected cycle $\widehat{C}_{2k}$, then $(H(D),\sigma)$ contains its corresponding negative cycle of the form $C_{2k}$ and so $(H(D),\sigma)$ is unbalanced signed graph. By Lemma $2.8$, $H(D)$ and $(H(D),\sigma)$ have different eigenvalues. By Lemma $3.8$, $D$ and $(D,\sigma)$ have different singular values, a contradiction.\\
Conversely, if $D$ does not contain semidirected cycle of the form $\widehat{C}_{2k}$ for some $k\ge 2$, then by Theorem $3.11$, the bipartite support graph $H(D)$ of $D$ is a forest. Then for any signature $\sigma$ on $D$, the associated bipartite signed graph of $(D,\sigma)$ is a signed forest with underlying graph $H(D)$. As $H(D)$ and $(H(D),\sigma)$ are switching equivalent. By Theorem $3.9$, there exist signature matrices $R$ and $C$ such that $RAC=A_{\sigma}$. By Lemma $3.10$, we see $A$ and $A_{\sigma}$ have same singular values i.e., $D$ and $(D,\sigma)$ have same adjacency singular spectrum.\\
If all semidirected cycles of the form $\widehat{C}_{2k}$ in $(D,\sigma)$ are positive, then the corresponding cycles of the form $C_{2k}$ in $(H(D),\sigma)$ are also positive and so $(H(D),\sigma)$ is balanced and hence $H(D)$ and $(H(D),\sigma)$ are switching equivalent. Again by Theorem $3.9$ and Lemma $3.10$, we see $D$ and $(D,\sigma)$ have same singular spectrum. This completes the proof.\qed
Following is an immediate consequence of Theorem $3.12$.
\begin{corollary}
Let $T\in \mathcal{T}(n)$. Then any two signed directed trees on $T$ have the same singular values.
\end{corollary}
{\bf Proof 1.} Let $(T,\tau)$ and $(T,\sigma)$ be any two signed directed trees on $T\in \mathcal{T}(n)$. Then $(H(D),\tau)$ and $(H(D),\sigma)$ are acyclic signed graphs with underlying graph $H(D)$. Clearly, $(H(D),\tau)$ and $(H(D),\sigma)$ are switching equivalent. Apply Theorem $3.9$ and Lemma $3.10$, the result follows.\\ 
{\bf Proof 2.} Given a directed tree $T\in\mathcal{T}(n)$, let $\mathcal{S}(T)$ denote the collection of all signed directed trees on $T$. We will show that any signed directed tree $K\in \mathcal{S}(T)$ is switching equivalent to $T$, the all positive directed tree. Let the $n$ vertices of $K$ be $v_{1}, v_{2}, \dots, v_{n}$. Define $\chi(v_{1}) = 1$ and $\chi(v_{i}) =$ sign of semi directed path from  $v_{1}$ to $v_{i}$ in $K$. 
 Then $K^{\chi}$ is the all positive signed directed tree $T$. 
 Accordingly, 
 $$M^{-1}A(K)M = A(T),$$ where $M=\text{diag}(\chi(v_1),\chi(v_2),\dots,\chi(v_n))$ is the signature matrix. Clearly, $A(K)$ and $A(T)$ have the same singular values.\qed
\section{\bf A lower bound for the trace norm of signed digraphs}\label{sec2}

The rank of a signed digraph $S$ is the number of positive singular values of its adjacency matrix. For digraphs, the singular values and rank have been studied in \cite{amr,zxw}. In \cite{wd}, the authors produce classification of signed digraphs with rank $2,3$ but considered a different matrix namely Eisenstein matrix. Our matrix is the usual adjacency matrix here. We first provide characterization of signed digraphs with rank one and use this to study the equality case in lower bound for the trace norm of signed digraphs. From here onwards, we assume our signed digraphs have no isolated vertices. The following result characterizes signed digraphs with rank $1$.
\begin{theorem}
Let $S = (D,\sigma)$ be a signed digraph of order $n \geq 2$. Then $\text{Rank}(S) = 1$
if and only if $S = (\overrightarrow{K}_{p,q}, \sigma)$ not containing  $(\overrightarrow{K}_{2,2}, \sigma)$ as a signed subdigraph with the negative sign.
\end{theorem}
{\bf Proof.} For $n=2$, the possible signed digraphs are $ (\overrightarrow{P}_{2}, \sigma)$ and $ (\overrightarrow{C}_{2}, \sigma)$. As  $\mathcal{S}((\overrightarrow{P}_{2}, \sigma))=\{1,0\}$ and $\mathcal{S}((\overrightarrow{C}_{2}, \sigma))= \{1^{(2)}\}$. Therefore, $ (\overrightarrow{P}_{2}, \sigma)$ is the only signed digraph on two vertices with rank $1$.\\
Assume $S=(D(\mathcal{V},\mathcal{A}), \sigma)$ has $n\ge 3$ vertices and $\text{Rank}(S)=1$ i.e., $A(S)$ has only one linearly independent row.\\ \textbf{Claim 1}. $S$ cannot have two signed arcs of the form $uv$,$vw$ where $\{u,v,w\} \subset \mathcal{V}$.\\ \textbf{Proof of claim 1}. Suppose signed digraph $S$ has arcs of the form $uv$, $vw$, where $u\ne v\ne w$. Then adjacency matrix $A(S)$ of $S$ has a $2\times2$ submatrix of the form
\[
\begin{bmatrix}
\sigma(uv) &*\\
0 &\sigma(vw)
\end{bmatrix},~~
\text{where}~\sigma(uv),\sigma(vw)\in\{-1,+1\}.
\]
So, $\text{Rank}(S)\ge 2$, a contradiction. This proves the claim.\\
Hence each vertex of $D$ is either a sink or a source vertex and so by Lemma $2.2$, $D$ is bipartite. Assume the partite sets of $D$ have cardinality $p$ and $q$ so that $p+q=n$.\\ 
\textbf{Claim 2}. The underlying digraph $D$ of $S$ is not a proper subdigraph of $\overrightarrow K_{p,q}$. \\
 \textbf{Proof of claim 2}.
If $D$ is a proper subdigraph of $\overrightarrow K_{p,q}$ with partite sets $V_1=\{v_1,v_2,\cdots,v_p\}$ and $V_2=\{u_1,u_2,\cdots,u_q\}$, then some vertex of $V_1$ is not adjacent to some vertex of $V_2$ say arc $v_1u_1\notin\mathcal{A}(S)$. Since $u_1$ is not an isolated vertex, therefore for some $v_i$, where $i=2,3,\dots,n$ say $v_2$, $v_2u_1\in \mathcal{A}(S)$. Then the rows indexed by $v_1$ and $v_2$ of the adjacency matrix $A(S)$ of $S$ are linearly independent and hence $\text{Rank}(S)$ is at least $2$. This proves $D=\overrightarrow K_{p,q}$ with $p+q=n$.
Therefore, the underlying digraph $D$ of $S$ is $\overrightarrow K_{p,q}$ with $p+q=n$. \\
\textbf{Claim 3}. $S=(\overrightarrow K_{p,q},\sigma)$ does not contain signed subdigraph $\overrightarrow K_{2,2}$ with negative sign.\\
\textbf{Proof of claim 3}. If $S$ contains $\overrightarrow K_{2,2}$ with negative sign, then $A(S)$ has a submatrix of order $2$ with entries from$\{-1,+1\}$ and having odd number of $-1$'s. It is easy to see that such a submatrix has two positive singular values and consequently by interlacing property Lemma $2.1$, $S$ has at least two positive singular values and hence $\text{Rank}(S)\ge 2$.  This proves the claim $3$.\\
Conversely, if $S = (\overrightarrow{K}_{p,q}, \sigma)$ not containing  $(\overrightarrow{K}_{2,2}, \sigma)$ as a signed subdigraph with the negative sign, then all $2\times 2$ minors are zero and hence all minors of order greater than $2$ of $A(S)$ are also zero. So, Rank (S)=1. \qed
Given two distinct vertices $u$ and $v$ in a signed digraph, by $\widehat{P}_2(u,v)$ we denote the semidirected path of length $2$ with arcs $uw$ and $vw$ and by $\widecheck{P}_2(u,v)$, we denote the semidirected path of length $2$ having arcs $wu$ and $wv$, where $w\neq u,v$.\\
In the next result, we consider the problem of characterizing signed digraphs with one singular value.
\begin{lemma} Let $S=(D,\tau)$ be a signed digraph of order $n$. Then $S$ has just one singular value i.e.,  $\sigma_1(S)=\sigma_2(S)=\dots=\sigma_n(S)=\sigma(S)$ if and only if $D$ is $\sigma^2(S)$ regular and for any two distinct vertices $v_i$ and $v_j$ in $S$, the number of positive $\widehat{P}_2(v_i,v_j)$ semidirected paths is equal to the number of negative $\widehat{P}_2(v_i,v_j)$ semidirected paths and the number of positive paths $\widecheck{P}_2(v_i,v_j)$ is equal to the number of negative $\widecheck{P}_2(v_i,v_j)$ paths. 
\end{lemma}
{\bf Proof.} If $S$ is discrete, then the result follows by Lemma $2.6$. We now assume that $S$ has at least one arc. let $A=A(S)$ be the adjacency matrix of $S$. By singular value decomposition, there is a real orthogonal matrix say $W$ such that $A=\sigma(S)W$. Accordingly $AA^T=A^TA=\sigma^2(S)I_n$. Therefore, $W$ is a weighing matrix of order $n$ and weight $\sigma^2(S)>0$.\\
For $i=j$, we see $[AA^T]_{ii}=\sigma^2(S)$ for all $i=1,2,\dots,n$. That is $d^+_i=\sigma^2(S)$ for all $i=1,2,\dots,n$. Again, $A^TA=\sigma^2(S)I_n$ gives $d^-_i=\sigma^2(S)$ for all $i=1,2,\dots,n$. Therefore $d^+_i=d^-_i=d ~(\text{say})=\sigma^2(S)$. Thus underlying digraph $D$ of $S$ is $\sigma^2(S)$ regular.\\
Further $[AA^T]_{ij}=\sum_{k=1}^n[A]_{ik}[A^T]_{kj}=\sum_{k=1}^n[A]_{ik}[A]_{jk}=0$ gives for any two distinct vertices say $v_i$ and $v_j$ of $S$, the number of positive $\widehat{P}_2(v_i,v_j)$ paths is equal to the number of negative  $\widehat{P}_2(v_i,v_j)$ paths. Similarly, for $i\ne j$,  $[A^TA]_{ij}=0$ gives the number of positive $\widecheck{P}_2(v_i,v_j)$ paths is equal to the number of negative  $\widecheck{P}_2(v_i,v_j)$ paths.\\
Conversely, suppose $S=(D,\tau)$ is a signed digraph such that $D$ is $\sigma^2(S)$ regular and for any two distinct vertices say $v_i$ and $v_j$ of $S$ the number of positive  $\widehat{P}_2(v_i,v_j)$ paths is equal to the number of negative $\widehat{P}_2(v_i,v_j)$ paths and the number of positive $\widecheck{P}_2(v_i,v_j)$ paths is equal to the number of negative $\widecheck{P}_2(v_i,v_j)$ paths. Then clearly $[AA^T]_{ij}=0=[A^TA]_{ij}$ for $i\ne j$ and $[AA^T]_{ii}=[A^TA]_{ii}=\sigma^2(S)$ for all $i=1,2,\dots,n$. So that $AA^T=A^TA=\sigma^2(S)I_n$. This shows that $S$ has  single singular values $\sigma(S)$. This completes the proof.\qed
In the next result, we obtain a lower bound for the trace norm of signed digraphs in terms of number of vertices and arcs.
\begin{theorem}
Let $S=(D,\sigma)$ be a signed digraph with $n \geq 2$ vertices, $m$ arcs and having no isolated vertices. Let $A$ be the adjacency matrix of $S$ and  $\sigma_1(S) \geq \sigma_2(S) \geq \dots \geq \sigma_n(S) \geq 0$ be the singular values of $S$.  Then\\
\begin{equation}
 \|S\|_* \;\geq\; \sqrt{\,m \;+\; n(n-1)\, |\det A(S)|^{2/n}} \,,
 \end{equation}

\medskip
with equality if and only if  
\begin{enumerate}
    \item $S = (\overrightarrow{K}_{p,q}, \sigma)$ not containing those $(\overrightarrow{K}_{2,2}, \sigma)$ having negative sign.
    \item $D$ is regular and for any two distinct vertices $u$ and $v$ in $S$, the number of positive $\widehat{P}_2(u,v)$ paths is equal to the number of negative $\widehat{P}_2(u,v)$ paths and the number of positive $\widecheck{P}_2(u,v)$ paths is equal to the number of negative $\widecheck{P}_2(u,v)$ paths.

\end{enumerate}

\end{theorem}
{\bf Proof.} 
Let $A$ be the adjacency matrix and $\sigma_1(S), \sigma_2(S), \dots,\sigma_n(S)$ be the singular values of $S$. We have the Frobenius norm
\[
\|S\|_F^2=\sum_{i,j} |a_{ij}|^2 = m.
\]
Also, $\|S\|_F^2=\sum_{i=1}^n \sigma_i^2(S)$ and $\prod_{i=1}^n \sigma_i(S) = |\det A(S)|$.

We have 
\begin{equation*}
\begin{aligned}
\|S\|_*^2&=\Big(\sum_i \sigma_i(S)\Big)^2\\
&=\sum_i \sigma_i(S)^2 + 2\sum_{i<j} \sigma_i(S)\sigma_j(S)\\
&= \|S\|_F^2 + 2\sum_{i<j} \sigma_i(S)\sigma_j(S) \\
&= m + 2\sum_{i<j} \sigma_i(S)\sigma_j(S).
\end{aligned}
\end{equation*}

By AM--GM inequality applied to the $\binom{n}{2}$ products, we have
\begin{equation*}
\begin{aligned}
\frac{1}{\binom{n}{2}}\sum_{i<j} \sigma_i(S) \sigma_j(S)\; &\ge\;
\Bigg(\prod_{i<j} \sigma_i(S) \sigma_j(S)\Bigg)^{\!1/\binom{n}{2}}\\
&= \Big( \prod_{i=1}^n \sigma_i(S) \Big)^{\!(n-1)/\binom{n}{2}}\\
&= |\det A(S)|^{\,2/n}.
\end{aligned}
\end{equation*}

Therefore, 
\[
\|S\|_* \;\ge\; \sqrt{\,m + n(n-1)\,|\det A(S)|^{2/n}}.
\]

Assume the equality holds in $(4.1)$, then  equality hold in AM--GM inequality applied on the collection
$\{\sigma_i(S) \sigma_j(S) : 1\le i<j\le n\}$, gives $\sigma_i(S)\sigma_j(S)= c \quad \text{for all } i<j$.\\
Two cases arise here.\\
\textbf{Case 1.} If $\sigma_1(S)\sigma_2(S) > 0$, then $\sigma_1(S)=\sigma_2(S)=\dots=\sigma_n(S)>0$.  
By Lemma $4.2$, the result follows in this case.\\  
\textbf{Case 2.} If $\sigma_1(S) > 0$, $\sigma_2(S)=0$, then $\sigma_2(S)=\sigma_3(S)=\dots=\sigma_n(S)=0$. Therefore, $\text{Rank}(S)=1$. By Theorem $4.1$, $S = (\overrightarrow{K}_{p,q}, \sigma)$ not containing those $(\overrightarrow{K}_{2,2}, \sigma)$ having negative sign.\\  
Conversely,  if $S = (\overrightarrow{K}_{p,q}, \sigma)$ not containing those $(\overrightarrow{K}_{2,2}, \sigma)$ having negative sign, then by Theorem $4.1$, $\text{Rank}(S)=1$ and so $S$ has only one positive singular value. Recall that sum of squares of the singular values is equal to the number of arcs, we see $\sigma_1(S)=\sqrt{pq}$ and $\sigma_i(S)=0$ for $i=2,3,\dots,n$. It is easy to see both sides of $(4.1)$ are equal to $\sqrt{pq}$.\\
If $D$ is  $d$ regular and for any two distinct vertices $u$ and $v$ in $S$, the number of positive $\widehat{P}_2(u,v)$  paths is equal to the number of negative $\widehat{P}_2(u,v)$  paths and the number of positive paths $\widecheck{P}_2(u,v)$ is equal to the number of negative $\widecheck{P}_2(u,v)$ paths, then by Lemma $4.2$, $\sigma_1(S) = \sigma_2(S) =\dots = \sigma_n(S)= \sqrt{d}$. It is easy to see that both sides of $(4.1)$ are equal to $n\sqrt{d}$. This completes the proof.\qed
The next result is an immediate consequence of Theorem $4.3$ and determines signed directed trees with minimum trace norm.
\begin{corollary}
Let $T\in (\mathcal{T}(n), \sigma)$. Then $$\|T\|_*\ge \sqrt{n-1},$$ 
with equality if and only if  $T=(\overrightarrow{K}_{1,n-1},\sigma)~ \text{or}~ T=(\overrightarrow{K}_{n-1,1},\sigma).$ \end{corollary}
\begin{example} Let $S=(D,\tau)$ be a union of signed directed cycles and let order of $S$ be $n$. It is clear that $S$ is $1$ regular and satisfies structural condition stated in case $2$ of Theorem $4.3$. It is easy to see that $H(S)$ is a signed matching of size $n$ and has spectrum $\{-1^{(n)},1^{(n)}\}$.  Accordingly, $\mathcal{S}(S)=\{1^{(n)}\}$. This provides a family of signed digraphs satisfying condition $2$ stated in Theorem $4.3$. Apart from this Hou et al. \cite{htw} provided examples of signed graphs with two distinct eigenvalues of equal modulus and hence their symmetric signed digraphs have all singular values equal. The problem of determining signed digraphs with one singular value actually is equivalent to determine  signed digraphs whose adjacency matrix is a weighing matrix with  zero diagonal.
\end{example}
\section{\bf Concluding remarks and future directions}
 We see from Theorem $4.3$ that if a signed digraph $S=(D,\tau)$ with atleast one arc, has all singular values $\sigma_i(S)$, $i=1,2,\dots,n$ equal. Then $D$ is regular and for any two distinct vertices $u$ and $v$ in $S$, the number of positive $\widehat{P}_2(u,v)$ paths is equal to the number of negative $\widehat{P}_2(u,v)$ paths and the number of positive $\widecheck{P}_2(u,v)$ paths is equal to the number of negative $\widecheck{P}_2(u,v)$ paths. At this moment, we are not able to determine such signed digraphs. So, we conclude this paper with the following problem for future study. 
\begin{question}\upshape Determine all signed digraphs $S=(D,\tau)$ such that $D$ is regular and for any two distinct vertices $u$ and $v$ in $S$, the number of positive $\widehat{P}_2(u,v)$ paths is equal to the number of negative $\widehat{P}_2(u,v)$ paths and the number of positive $\widecheck{P}_2(u,v)$ paths is equal to the number of negative $\widecheck{P}_2(u,v)$ paths.
\end{question}  

\noindent{\bf Acknowledgements.} The research of   M. A. Bhat is supported by NBHM project No. 02011/42\\/2025/NBHM(RP)/R\&DII/16567 and by SERB-DST grant with File No. MTR/2023/000201.

\end{document}